\documentclass[11pt]{article}
\usepackage{amssymb} 
\usepackage{graphicx} 
\usepackage{amsthm} 
\usepackage{color}
\usepackage{amsmath}
\usepackage{hyperref}
\usepackage{bbm}
\usepackage{dsfont}
\usepackage{filecontents}
\usepackage[a4paper, total={6in, 8in}]{geometry}
\usepackage{filecontents}
\usepackage{cleveref}

\theoremstyle{plain} 
\newtheorem{prop}{Proposition}[section]
\newtheorem{lemma}{Lemma}[section]
\newtheorem{thrm}{Theorem}[section]

\theoremstyle{definition} 
\newtheorem{mydef}{Definition}[section]
\newtheorem{assumption}{Assumption}

\theoremstyle{remark} 
\newtheorem{remark}{Remark}[section]

\numberwithin{equation}{section}

\title{Random walks in random conductances: decoupling and spread of infection}
\author{Peter Gracar\footnote{Department of Mathematical Sciences, University of Bath, UK. {\tt p.gracar@bath.ac.uk}}\; and Alexandre Stauffer\footnote{Department of Mathematical Sciences, University of Bath, UK, {\tt a.stauffer@bath.ac.uk}. Supported by a Marie Curie Career Integration Grant PCIG13-GA-2013-618588 DSRELIS, and an EPSRC Early Career Fellowship.}}
\date{}

\begin{document}
\maketitle

\begin{abstract}
Let \((G,\mu)\) be a \emph{uniformly elliptic} random conductance graph on \(\mathbb{Z}^d\) with a Poisson point process of particles at time \(t=0\) that perform independent simple random walks. We show that inside a cube \(Q_K\) of side length \(K\), if all subcubes of side length \(\ell<K\) inside \(Q_K\) have sufficiently many particles, the particles return to stationarity after \(c\ell^2\) time with a probability close to \(1\). We also show this result for percolation clusters on locally finite graphs. Using this mixing result and the results of \cite{Gracar2016a}, we show that in this setup, an infection spreads with positive speed in any direction. Our framework is robust enough to allow us to also extend the result to infection with recovery, where we show positive speed and that the infection survives indefinitely with positive probability.
\end{abstract}

\section{Introduction}\label{sect:intro}

We consider the graph \(G=(\mathbb{Z}^d,E)\), \(d\geq 2\) to be the \(d\)-dimensional square lattice, with edges between nearest neighbors: for \(x,y\in\mathbb{Z}^d\) we have \((x,y)\in E\) iff \(\|x-y\|_1=1\). Let \(\{\mu_{x,y}\}_{(x,y)\in E}\) be a collection of i.i.d.\ non-negative weights, which we call \emph{conductances}. In this paper, edges will always be undirected, so \(\mu_{x,y}=\mu_{y,x}\) for all \((x,y)\in E\). We also assume that the conductances are uniformly elliptic: that is,
\begin{equation}\label{eq:mu_bounds_new}
	\textrm{there exists \(C_M>0\), such that }\mu_{x,y}\in[C_M^{-1},C_M]\textrm{ for all }(x,y)\in E.
\end{equation}

We say \(x\sim y\) if \((x,y)\in E\) and define \(\mu_x=\sum_{y\sim x}\mu_{x,y}\). At time \(0\), consider a Poisson point process of particles on \(\mathbb{Z}^d\), with intensity measure \(\lambda(x)=\lambda_0\mu_x\) for some constant \(\lambda_0>0\) and all \(x\in\mathbb{Z}^d\). That is, for each \(x\in\mathbb{Z}^d\), the number of particles at \(x\) at time \(0\) is an independent Poisson random variable of mean \(\lambda_0\mu_x\). Then, let the particles perform independent continuous-time simple random walks on the weighted graph so that a particle at \(x\in\mathbb{Z}^d\) jumps to a neighbor \(y\sim x\) at rate \(\frac{\mu_{x,y}}{\mu_x}\). It follows from the thinning property of Poisson random variables that the system of particles is in stationarity; thus, at any time \(t\), the particles are distributed according to a Poisson point process with intensity measure \(\lambda\).

We study the spread of an infection among the particles. Assume that at time \(0\) there is an infected particle at the origin, and all other particles are uninfected. Then an uninfected particle gets infected as soon as it shares a site with an infected particle. Our first result establishes that the infection spreads with positive speed.

\begin{thrm}\label{thrm:spread}
	Let \(\{\mu_{x,y}\}_{(x,y)\in E}\) be i.i.d.\ satisfying (\ref{eq:mu_bounds_new}). For any time \(t\geq 0\), let \(I_t\) be the position of the infected particle that is furthest away from the origin. Then
	\[
		\liminf_{t\rightarrow\infty}\frac{\|I_t\|_1}{t}>0\quad \textrm{almost surely}.
	\]
\end{thrm}

The above result has been established on the square lattice (i.e., \(\mu_{x,y}=1\) for all \((x,y)\in E\)) by Kesten and Sidoravicius \cite{Kesten2005} via an intricate multi-scale analysis; see also \cite{Kesten2008} for a shape theorem. In a companion paper \cite{Gracar2016a}, we develop a framework which can be used to analyze processes on this setting without the need of carrying out a multi-scale analysis from scratch. We prove our \Cref{thrm:spread} via this framework, showing the applicability of our technique from \cite{Gracar2016a}. We also apply this technique to analyze the spread of an infection with recovery.
Let the setup be as before, but now each infected particle independently recovers and becomes uninfected at rate \(\gamma\) for some fixed parameter \(\gamma>0\). After recovering, a particle becomes again susceptible to the infection and gets infected again whenever it shares a site with an infected particle. Our next result shows that if \(\gamma\) is small enough, then with positive probability there will be at least one infected particle at all times. When this happens, we also obtain that the infection spreads with positive speed.

\begin{thrm}\label{thrm:spread_recov}
	Let \(\{\mu_{x,y}\}_{(x,y)\in E}\) be i.i.d.\ satisfying (\ref{eq:mu_bounds_new}). For any \(\lambda_0>0\), there exists \(\gamma_0>0\) such that, for all \(\gamma\in(0,\gamma_0)\), with positive probability, the infection does not die out. Furthermore, there are constants \(c_1,c_2>0\) such that
	\[
	\mathbb{P}[\|I_t\|_1\geq c_1t\textrm{ for all }t\geq0]\geq c_2,
	\]
	where \(I_t\) is the position of the infected particle that is furthest away from the origin at time~\(t\).
\end{thrm}

The challenge in this setup comes from the heavily dependent structure of the model. Though particles move independently of one another, dependencies do arise over time. For example, if a ball of radius \(R\) centered at some vertex \(x\) of the graph turns out to have no particles at time 0, then the ball \(B(x,R/2)\) of radius \(R/2\) centered at \(x\), will continue to be empty of particles up to time \(R^2\), with positive probability. This means that the probability that the \((d+1)\)-dimensional, space-time cylinder \(B(x,R/2)\times[0,R^2]\) has no particle is at least \(\exp\{-cR^d\}\) for some constant \(c\), which is just a stretched exponential in the volume of the cylinder.
On the other hand, one expects that, after time \(t\gg R^2\), the set of particles inside the ball will become ``close'' to stationarity.

To deal with dependences, one often resorts to a decoupling argument, showing that two local events behave roughly independently of each other, provided they are measurable according to regions in space time that are sufficiently far apart. We will obtain such an argument by extending a technique which we call \emph{local mixing}, and which was introduced in \cite{Sinclair2010}. The key observation is the following. Consider a cube \(Q\subseteq\mathbb{Z}^d\), tessellated into subcubes of side length \(\ell>0\). For simplicity assume for the moment that \(\mu_{x,y}=1\) for all \((x,y)\in E\). Suppose that at some time \(t\), the configuration of particles inside \(Q\) is dense enough, in the sense that inside each subcube there are at least \(c\ell^d\) particles, for some constant \(c>0\). Regardless of how the particles are distributed inside \(Q\), as long as the subcubes are dense, we obtain that at some time \(t+c'\ell^2\), not only particles had enough time to move out of the subcubes they were in at time \(t\), but also we obtain that the configuration of particles inside ``the core'' of \(Q\) (i.e., away from the boundary of \(Q\)) stochastically dominates a Poisson point process of intensity \((1-\epsilon)c\ell^d\) that is independent of the configuration of particles at time \(t\). Moreover, the value \(\epsilon\) can be made arbitrarily close to \(0\) by setting \(c'\) large enough. In words, we obtain a configuration at time \(t+c'\ell^2\) inside the core of \(Q\) that is roughly independent of the configuration at time \(t\), and is close to the stationary distribution. To the best of our knowledge, the idea of local mixing in such settings originated in the work of Sinclair and Stauffer \cite{Sinclair2010}, and was later applied in \cite{Peres2012,Stauffer2014}. This idea was then extended  with the introduction of soft local times by Popov and Teixeira \cite{Popov2012} (see also \cite{Hilario2014}), and applied to other processes, such as random interlacements.

Our second main goal in this paper is to show that this local mixing result can be obtained in a larger setting, in which a local CLT result, which plays a crucial role in the proof\footnote{The results of \cite{Sinclair2010,Peres2012} are in the setting of Brownian motions on \(\mathbb{R}^d\), but can be adapted in a straightforward way to random walks on \(\mathbb{Z}^d\) with \(\mu_{x,y}=1\) for all \((x,y)\in E\) by using the local CLT.} of \cite{Sinclair2010,Peres2012,Hilario2014}, might not hold or only holds in the limit as time goes to infinity, with no good control on the convergence rate. This is precisely the situation of our setting, where the weights \(\mu_{x,y}\) are not all identical to \(1\). To work around that, we will show that local mixing can be obtained whenever a so-called \emph{Parabolic Harnack Inequality} holds, and we have some good estimates on the displacement of random walks. 
 
For the result below, we can impose slightly weaker conditions on \(\mu_{x,y}\). Let \(p_c\) be the critical probability for bond percolation on \(\mathbb{Z}^d\). Assume that \(\mu_{x,y}\) are i.i.d.\ and that, for each \((x,y)\in E\), we have
\begin{equation}\label{eq:p_c_new}
	\mathbb{P}[\mu_{x,y}=0]< p_c\textrm{ and }\mu_{x,y}\textrm{ satisfies (\ref{eq:mu_bounds_new}) whenever }\mu_{x,y}>0.
\end{equation}

For two regions \(Q'\subseteq Q\subset\mathbb{Z}^d\), we say that \(Q'\) is \(x\) away from the boundary of \(Q\) iof the distance between \(Q'\) and \(Q^c\) is at least \(x\). 
\begin{thrm}\label{thrm:mixing_simple}
	Let \(\{\mu_{x,y}\}_{(x,y)\in E}\) be i.i.d.\ satisfying (\ref{eq:p_c_new}). There exist positive constants \(c_1\), \(c_2\), \(c_3\), \(c_4\), \(c_5\) such that the following holds.
	Fix \(K>\ell>0\) and \(\epsilon\in(0,1)\). Consider a cube \(Q\) of side-length \(K\), tessellated into subcubes \((T_i)_i\) of side length \(\ell\).
	Assume each subcube \(T_i\) contains at least \(\beta\sum_{x\in T_i}\mu_x\) particles for some \(\beta>0\), and let \(\Delta\geq c_1\ell^2\epsilon^{-c_2}\).
	If \(\ell\) is large enough, then after the particles move for time \(\Delta\), we obtain that within a region \(Q'\subseteq Q\) that is at least \(c_3\ell\epsilon^{-c_4}\) away from the boundary of \(Q\), the particles dominate an independent Poisson point process of intensity measure \(\nu(x)=(1-\epsilon)\beta\mu_x\), \(x\in Q'\), with probability at least
	\[
		1-\sum_{y\in Q'}\exp\left\{-c_5\beta\mu_y\epsilon^2\Delta^{d/2}\right\}.
	\]
\end{thrm}

We will prove a more detailed version of this theorem in \Cref{sect:mixing} (see \Cref{thrm:mixing}). Although we only prove the result for the case of conductances on the square lattice, \Cref{thrm:mixing_simple} holds for more general graphs. The theorem holds for any graph \(G\) and any region \(Q\) of \(G\) that can be tessellated into subregions of diameter at most \(\ell\) whenever each such subregion is dense enough, the so-called parabolic Harnack inequality holds for \(G\) and we have estimates on the displacement of random walks on \(G\). We discuss some extensions in \Cref{sect:extensions}.

The structure of this paper is as follows. In \Cref{sect:setup}, we formally define the family of graphs we consider for local mixing and present results concerning the parabolic Harnack inequality, heat kernel bounds and exit times for random walks on such graphs. In \Cref{sect:mixing}, we state a more precise version of \Cref{thrm:mixing_simple} and prove it. In \Cref{sect:extensions} we prove an extension of the local mixing result to random walks whose displacement is conditioned to be bounded, which is particularly useful in applications \cite{Sinclair2010,Gracar2016a}. In \Cref{sect:spread}, we use the local mixing result and results from our companion paper \cite{Gracar2016a} to prove Theorems \ref{thrm:spread} and \ref{thrm:spread_recov} for graphs satisfying (\ref{eq:mu_bounds_new}).

\section{Heat kernel estimates and exit times}\label{sect:setup}

In this section, we consider a general graph \(G=(V,E)\), with uniformly bounded degrees. For \(x,y\in V\), let \(|x-y|\) denote the distance between \(x\) and \(y\) in \(G\). For \(x\in V\), let \(B(x,r)=\{y\in V:\;|x-y|\leq r\}\) be the ball of radius \(r\) centered at \(x\). We consider non-negative weights (conductances) \((\mu_{x,y})_{(x,y)\in E}\), that are symmetric. As in \Cref{sect:intro}, we denote by \(x\sim y\) whenever \(x,y\in V\) are neighbors in \(G\), and define \(\mu_x=\sum_{y\sim x}\mu_{x,y}\). We also extend \(\mu\) to a measure on \(V\). For simplicity, the reader may think of \(V\) as \(\mathbb{Z}^d\) and \(\mu_{x,y}\) being i.i.d.\ random variables satisfying (\ref{eq:p_c_new}).
We keep our notation in greater generality as we want to highlight the exact conditions we need for our results.

Assume the existence of \(d\geq 1\) and \(C_U\) such that
\begin{equation}\label{for:graphupperbound}
	\mu(B(x,r))\leq C_Ur^d,\quad \textrm{for all }r\geq1,\textrm{ and }x\in V.
\end{equation}

We consider a continuous time simple random walk on the weighted graph \(\mathcal{G}:=(G,\mu)\), which jumps from vertex \(x\) to vertex \(y\) at rate \(\frac{\mu_{x,y}}{\mu_x}\). More formally, for any function \(f:V\rightarrow\mathbb{R}\), let
\begin{equation}\label{for:heatequation}
	\mathcal{L}f(x)=\mu_x^{-1}\sum_{y\sim x}\mu_{x,y}(f(y)-f(x)),
\end{equation}
and define the random walk started at vertex \(x\) as the Markov process \(Y=(Y_t,t\in[0,\infty),\mathbb{P}_x,x\in V)\) with generator \(\mathcal{L}\). Its \emph{heat kernel} on the graph is defined as
	\begin{equation}\label{for:heatkernel}
		q_t(x,y)=\frac{\mathbb{P}_x(Y_t=y)}{\mu_y},\textrm{ for any }x,y\in V.
	\end{equation}	
						
	We will say that a particle walks along \(\mathcal{G}\) if it is a Markov process with generator \(\mathcal{L}\) as defined above.
	We now state several definitions from \cite{Barlow2004} which we use throughout the paper.
						
	\begin{mydef}[Very good balls]\label{def:balls}
		Let \(C_V\), \(C_P\) and \(C_W\geq 1\) be fixed constants. We say \(B(x,r)\) is \((C_V,C_P,C_W)-good\) if:
		\[
			\mu(B(x,r))\geq C_Vr^d,
		\]
		and the weak Poincar\'e inequality
		\[
			\sum_{y\in B(x,r)}(f(y)-\bar f_{B(x,r)})^2\mu_y\leq C_Pr^2\sum_{y,z\in B(x,C_Wr),z\sim y}(f(y)-f(z))^2\mu_{yz}
		\]
		holds for every \(f:B(x,C_Wr)\rightarrow\mathbb{R}\), where \(\bar f_{B(x,r)}=\mu(B(x,r))^{-1}\sum_{y\in B(x,r)}f(y)\mu_y\) is the weighted average of \(f\) in \(B(x,r)\). Furthermore, we say \(B(x,R)\) is \((C_V,C_P,C_W)- very\; good\) if there exists \(N_B=N_{B(x,R)}\leq R^{1/(d+2)}\) such that for all \(r\geq N_B\),  \(B(y,r)\) is good whenever \(B(y,r)\subseteq B(x,R)\). We assume that \(N_B\geq 1\).
	\end{mydef}
						
	For the remainder of the paper we assume that \(d\geq 2\), fix \(C_U\), \(C_V\), \(C_P\) and \(C_W\) and take \(\mathcal{G}=(V,E,\mu)\) to satisfy (\ref{for:graphupperbound}).
						
	We are now ready to present some key results from \cite{Hambly2009} that control the variation of the random walk density function. We will also present a result about random walk exit times which was initially shown in \cite{Barlow2004} for Bernoulli percolation clusters and then generalized to our setup in \cite{Hambly2009}. The first result gives Gaussian upper and lower bounds for the heat kernel for very good balls.
	
\begin{prop}\label{prop:heat_kernel_bounds}\cite[Theorem 2.2]{Hambly2009}
	Assume the weights \(\mu_{x,y}\) are i.i.d.\ and (\ref{eq:p_c_new}) holds. Fix a vertex \(x\in V\). Suppose there exists \(R_1=R_1(x)\) such that \(B(x,R)\) is very good with \(N_{B(x,R)}^{3(d+2)}\leq R\) for every \(R\geq R_1\). Then there exist positive constants \(c_1\), \(c_2\), \(c_3\), \(c_4\) such that if \(t\geq R_1^{2/3}\), we obtain
	\[
		q_t(x,y)\leq c_1t^{-d/2}e^{-c_2|x-y|^2/t},\quad\textrm{for all }y\in V\textrm{ with }|x-y|\leq t
	\]
	and
	\[
		q_t(x,y)\geq c_3t^{-d/2}e^{-c_4|x-y|^2/t},\quad\textrm{for all }y\in V\textrm{ with }|x-y|^{3/2}\leq t.
	\]
\end{prop}

Now define the space-time regions
	\begin{eqnarray*}
	Q(x,R,T)&=&B(x,R)\times (0,T],\\
	Q_-(x,R,T)&=&B(x,\tfrac{R}{2})\times[\tfrac{T}{4},\tfrac{T}{2}]\\
	&\textrm{and}&\\ 
	Q_+(x,R,T)&=&B(x,\tfrac{R}{2})\times[\tfrac{3T}{4},T].
\end{eqnarray*}
	
We denote by \(t+Q(x,R,T)=B(x,R)\times (t,t+T)\). We call a function \(u:V\times\mathbb{R}\rightarrow\mathbb{R}\) \emph{caloric} on \(Q\) if it is defined on \(Q=Q(x,R,T)\) and
\[
	\frac{\partial}{\partial t}u(x,t)=\mathcal{L}u(x,t)\quad\textrm{for all }(x,t)\in Q.
\]

We say the \emph{parabolic Harnack inequality} (PHI) holds with constant \(C_H\) for \(Q=Q(x,R,T)\) if whenever \(u=u(x,t)\) is non-negative and caloric on \(Q\), then
\[
	\sup_{(x,t)\in Q_-(x,R,T)}u(x,t)\leq C_H\inf_{(x,t)\in Q_+(x,R,T)}u(x,t).
\]

It is well known that the heat kernel of a random walk on \(\mathcal{G}\) started at \(x\) is a caloric function; in fact taking \(x=0\) and \(u(x,t)=q_t(0,x)\) we have
\begin{eqnarray*}
	\frac{d}{dt}q_t(0,x)&=&\lim_{dt\rightarrow0}\frac{1}{\mu_{x}}\frac{\sum_{y\in V}\mathbb{P}_0(Y_t=y)\mathbb{P}_y(Y_{dt}=x)-\mathbb{P}_0(Y_t=x)\mathbb{P}_x(Y_{dt}\neq x)}{dt}\\
	&=&\frac{1}{\mu_x}\left(\sum_{y\sim x}\mathbb{P}_0(Y_t=y)\frac{\mu_{y,x}}{\mu_y}-\sum_{y\sim x}\frac{\mu_{x,y}}{\mu_x}\mathbb{P}_0(Y_t=x)\right)\\
	&=&\frac{1}{\mu_x}\sum_{y\sim x}\mu_{x,y}(q_t(0,y)-q_t(0,x))=\mathcal{L}q_t(0,x).
\end{eqnarray*}

The main result from \cite{Hambly2009} shows that the PHI holds in regions that are very good according to \Cref{def:balls}.

\begin{prop}\label{thrm:PHI}	\cite[Theorem 3.1]{Hambly2009}	
	Let \(x_0\in V\). Suppose that \(R_1\geq 16\) and \(B(x_0,R_1)\) is \((C_V,C_P,C_W)- very\; good\) with \(N_{B(x_0,R_1)}^{2d+4}\leq R_1/(2\log R_1)\). Then there exists a constant \(C_H>0\) such that the PHI holds for \(Q(x_1,R,R^2)\) for any \(x_1\in B(x_0,R_1/3)\) and for \(R\) such that \(R\log R=R_1\).
\end{prop}

A direct consequence of the PHI is the following known proposition, which when applied to the caloric function \(u(x,t)=q_t(0,x)\) gives that \(q_t(0,x)\) and \(q_t(0,y)\) are very similar to each other when \(x\) and \(y\) are close by. This property will be crucial for our proof of local mixing, so we give the proof of this proposition for completeness.

\begin{prop}\label{prop:PHI}
	Let \(x_0\in V\). Suppose that there exists \(s(x_0)\geq 0\) so that for all \(R\geq s(x_0)\), the PHI holds with constant \(C_H\) for \(Q(x_0,R,R^2)\) and that the ball \(B(x_0,R)\) is \((C_V,C_P,C_W)- very\; good\). Let \(\Theta=\log_2(C_H/(C_H-1))\), and for \(x,y\in V\) define
	\[
		\rho(x_0,x,y)=s(x_0)\vee |x_0-x|\vee |x_0-y|.
	\]
	There exists a constant \(c>0\) such that the following holds. Let \(r_0\geq s(x_0)\) and suppose that \(u=u(x,t)\) is caloric in \(Q=Q(x_0,r_0,r_0^2)\). Then for any \(x_1,x_2\in B(x_0,\frac{1}{2}r_0)\) and any \(t_1,t_2\) such that \(r_0^2-\rho(x_0,x_1,x_2)^2\leq t_1,t_2\leq r_0^2\) we have
	\begin{equation}\label{for:phi}
		|u(x_1,t_1)-u(x_2,t_2)|\leq c\left(\frac{\rho(x_0,x_1,x_2)}{r_0}\right)^{\Theta}\sup_{(t,x)\in Q_+(x_0,r_0,r_0^2)}|u(t,x)|.
	\end{equation}
\end{prop}

\begin{proof}
	For any integer \(k\geq0\), set \(r_k=2^{-k}r_0\), and let
	\begin{eqnarray*}
		Q(k)&=&(r_0^2-r_k^2)+Q(x_0,r_k,r_k^2),\\
		Q_+(k)&=&(r_0^2-r_k^2)+Q_+(x_0,r_k,r_k^2)\\
		&and&\\
		Q_-(k)&=&(r_0^2-r_k^2)+Q_-(x_0,r_k,r_k^2).
	\end{eqnarray*}
	This gives that \(Q_+(k)=Q(k+1)\). Now take \(k\geq 1\) small enough, so that \(r_k\geq s(x_0)\). If we apply the PHI to the non-negative caloric functions \(-u+\sup_{Q(k)}u\) and \(u-\inf_{Q(k)}u\), we get the inequalities
	\begin{eqnarray*}
		\sup_{Q(k)}u-\inf_{Q_-(k)}u&\leq& C_H(\sup_{Q(k)}u-\sup_{Q_+(k)}u)
	\end{eqnarray*}
	and
	\begin{eqnarray*}
		\sup_{Q_-(k)}u-\inf_{Q(k)}u&\leq& C_H(\inf_{Q_+(k)}u-\inf_{Q(k)}u).
	\end{eqnarray*}
	Adding them together and using \(\sup_{Q_-(k)}u-\inf_{Q_-(k)}u\geq 0\) gives
	\[
		\sup_{Q(k)}u-\inf_{Q(k)}u\leq C_H(\sup_{Q(k)}u-\inf_{Q(k)}u)-C_H(\sup_{Q_+(k)}u-\inf_{Q_+(k)}u)
	\]

	Denoting by \(\operatorname{Osc}(u,A)=\sup_{A}u-\inf_{A}u\) and	setting \(\delta=C_H^{-1}\), this gives
							
	\begin{equation}\label{for:osc}
		\operatorname{Osc}(u,Q_+(k))\leq(1-\delta)\operatorname{Osc}(u,Q(k)).
	\end{equation}
							
	Next, take the largest \(m\) such that \(r_m\geq\rho(x_0,x_1,x_2)\). Then, applying (\ref{for:osc}) repeatedly on \(Q(1)\supset Q(2)\supset\dots Q(m)\) yields, since \((x_i,t_i)\in Q(m)\),
	\[
		|u(t_1,x_1)-u(t_2,x_2)|\leq\operatorname{Osc}(u,Q(m))\leq (1-\delta)^{m-1}\operatorname{Osc}(u,Q(1)).
	\]
	Since
	\[
		(1-\delta)^m=2^{-m\Theta}\leq \left(\frac{2\rho(x_0,x_1,x_2)}{r_0}\right)^{\Theta},
	\]
	the result follows.
\end{proof}

We will also need to control the exit time of the random walk out of a ball of radius \(r\), which we define as
\[
	\tau(x,r)=\inf\{t: Y_t\not\in B(x,r)\}.
\] 

\begin{prop}\label{prop:exit}
	Let \(x_0\in V\) and let \(B(x_0, R)\) be \((C_V,C_P,C_W)- very\; good\) with \(N_B^{d+2}<R\). Let \(x\in B(x_0,\frac{5}{9}R)\). There exist positive constants \(c_1\), \(c_2\), \(c_3\), \(c_4\) such that if \(t\), \(r\) satisfy
	\begin{equation}\label{for:exit_times_condition}
		0<r\leq R\quad\textrm{and}\quad c_1 N_B^d(\log N_B)^{1/2}r\leq t\leq c_{2} R^2/\log R,
	\end{equation}
	then we have
	\begin{equation}\label{for:exit_time_bound}
		\mathbb{P}_x(\tau(x,r)<t)\leq c_3\exp\{-c_4r^2/t\}.
	\end{equation}		
\end{prop}

\begin{proof}
	The proposition was proven for percolation clusters in \cite[Proposition 3.7]{Barlow2004}. The proof for more general \(\mathcal{G}\) is similar and can be found in \cite[Theorem 2.2a]{Hambly2009}.
\end{proof}

Since Propositions \ref{prop:heat_kernel_bounds}, \ref{thrm:PHI} and \ref{prop:exit} rely on very good balls and the related value \(N_B\), we can assume a lower bound \(S\) such that if \(R>S\), then the conditions of all three are satisfied. More formally, we assume the following. 

\begin{assumption}\label{ass:S}
	The graph \(G\) has polynomial growth; i.e., it satisfies (\ref{for:graphupperbound}). Furthermore, there exists a function \(S: V\mapsto \mathbb{R}\) such that for all \(R_1\) with \(R_1\log R_1\geq S(x_0)\), the ball \(B(x_0,R_1)\) is \((C_V,C_P,C_W)\)-very good with \(N_{B(x_0,R_1)}^{2d+4}\leq R_1\).
	As a consequence, Propositions \ref{prop:heat_kernel_bounds}, \ref{thrm:PHI}, \ref{prop:PHI} and \ref{prop:exit} all hold for any \(R>S(x)\).
\end{assumption}
For i.i.d.\ weights as defined in \Cref{sect:intro}, we obtain the following.
\begin{prop}\label{prop:A}
	If \(V=\mathbb{Z}^d\) and the weights \(\mu_{x,y}\) are i.i.d.\ and satisfy (\ref{eq:mu_bounds_new}) or (\ref{eq:p_c_new}), then \Cref{ass:S} holds. Furthermore, we have that there exist constants \(c,\gamma>0\) such that 
	\[
		\mathbb{P}[S(x)\geq n]\leq c\exp\{-cn^{\gamma}\}\textrm{ for all }x\in\mathbb{Z}^d\textrm{ and }n\geq 0.
	\]
	If the weights \(\mu_{x,y}\) are i.i.d.\ and satisfy (\ref{eq:mu_bounds_new}), then \Cref{ass:S} holds with \(S(x)= 1\) for all \(x\in V\).
\end{prop}

\begin{proof}
	This has been shown in \cite{Hambly2009}, following from the framework developed in \cite[Theorem 2.18 and Lemma 2.19]{Barlow2004}. The tail estimate of \(S(x)\) was obtained in \cite[Theorem 5.7]{Hambly2009}. For the case where \(\mu_{x,y}\) satisfy (\ref{eq:mu_bounds_new}), the function \(S\) can be set to \(1\) by \cite[Theorem 5.7]{Hambly2009} and the results of \cite{Delmotte1999}.
\end{proof}

\begin{remark}
	In \cite{Berger2008} it has been shown that when the weights \(\mu_{x,y}\) are i.i.d.\ but can assume values arbitrarily close to zero, so neither (\ref{eq:mu_bounds_new}) nor (\ref{eq:p_c_new}) hold, it is possible to find distributions (at least in dimensions \(d\geq 5\)) for which \Cref{ass:S} does not hold. Hence, even though we do not explicitly use uniform ellipticity of \(\mu_{x,y}\) in our proofs, this property has a fundamental role in our analysis. Recent results (see, for example, \cite{Andres2016}) have been derived to relax assumption (\ref{eq:p_c_new}), but they do not establish all properties we need.
\end{remark}

\section{Decoupling via local mixing}\label{sect:mixing}

In this section, we will restrict to the case \(V=\mathbb{Z}^d\) and \((x,y)\in E\) if and only if \(\|x-y\|_{1}=1\), but we do not assume the \(\mu_{x,y}\) are i.i.d. We define a cube of side length \(z>0\) as \(Q_z:=[-z/2,z/2]^d\).
In the remainder of the paper, we will work with the heat kernel \(q_t\) as defined in (\ref{for:heatkernel}). Since we allow \(\mu_{x,y}=0\), it is possible for two sites not to be connected. To address this we require the existence of an infinite component. Formally, we assume the following.

\begin{assumption}\label{ass:critical}
	For each \((x,y)\in E\), either \(\mu_{x,y}=0\) or it satisfies (\ref{eq:mu_bounds_new}) for a uniform constant \(C_M\). Moreover, the weights \(\mu_{x,y}\) are such that an infinite connected component of edges of positive weight within \(\mathcal{G}\) exists and contains the origin.
\end{assumption}
With this let \(\mathcal{C}_{\infty}\) be the infinite connected component of \(\mathcal{G}\) that contains the origin and define 
	\[
		\tilde Q_z:=Q_z\cap \mathcal{C}_{\infty}.
	\]
								
We note that if \(\mu_{x,y}\) satisfy (\ref{eq:mu_bounds_new}), then \Cref{ass:critical} is automatically satisfied. We will continue to call \(\tilde Q_z\) as a ``cube''. We are now ready to state the more detailed version of \Cref{thrm:mixing_simple}.			
				
\begin{thrm}\label{thrm:mixing}
	Let \(\mu_{x,y}\) satisfy Assumptions \ref{ass:S} and \ref{ass:critical}. There exist constants \(c_0\), \(c_1\), \(C>0\) such that the following holds.
	Fix \(K>\ell>0\) and \(\epsilon\in(0,1)\). Consider the cube \(Q_K\) tessellated into subcubes \((T_i)_{i}\) of side length \(\ell\) and assume that \(\ell>S^{d+1}(x)\) for all \(x\in \tilde Q_K\). 
	Let \((x_j)_{j}\subset \tilde Q_{K}\) be the locations at time \(0\) of a collection of particles, such that each subcube \(\tilde T_i\) contains at least \(\sum_{y\in \tilde T_i}\beta\mu_y\) particles for some \(\beta>0\).
	Let \(\Delta\geq c_0\ell^2\epsilon^{-4/\Theta}\) where \(\Theta\) is as in \Cref{prop:PHI}.
	For each \(j\) denote by \(Y_j\) the location of the \(j\)-th particle at time \(\Delta\). 
	Fix \(K'>0\) such that \(K-K'\geq\sqrt{\Delta}c_1\epsilon^{-1/d}\). Then there exists a coupling \(\mathbb{Q}\) of an independent Poisson point process \(\psi\) with intensity measure \(\zeta(y)=\beta(1-\epsilon)\mu_y\), \(y\in \mathcal{C}_{\infty}\), and \((Y_j)_{j}\) such that within \(\tilde Q_{K'}\subset \tilde Q_K\), \(\psi\) is a subset of \((Y_j)_{j}\) with probability at least
	\[
		1-\sum_{y\in \tilde Q_{K'}}\exp\left\{-C\beta\mu_y\epsilon^2\Delta^{d/2}\right\}.
	\]
\end{thrm}

				
Note that, due to \Cref{prop:A}, \Cref{thrm:mixing_simple} is a special case of \Cref{thrm:mixing}, which we prove below. In order to do so, we will use something called \emph{soft local times}, which was introduced in \cite{Popov2012} to analyze random interlacements, following the introduction of local mixing in \cite{Sinclair2010,Peres2012,Stauffer2014}; see also \cite{Hilario2014} for an application of this technique to random walks on \(\mathbb{Z}^d\).

\begin{prop}\label{prop:mixing}
	Let \((Z_j)_{j\leq J}\) be a collection of \(J\) independent random particles on \(V\) distributed according to a family of density functions \(g_j:V\rightarrow\mathbb{R}\), \(j\leq J\). Define for all \(y\in V\) the soft local time function \(H_J(y)=\sum_{j=1}^{J}\xi_jg_j(y)\), where the \(\xi_j\) are i.i.d.\ exponential random variables of mean \(1\). Let \(\psi\) be a Poisson point process on \(V\) with intensity measure \(\rho:V\rightarrow\mathbb{R}\) and define the event
	\(
		E=\left\{\psi\textrm{ is a subset of }(Z_j)_{j\leq J}\right\}.
	\)
	Then there exists a coupling such that,  
	\[
		\mathbb{P}\left[E\right]\geq\mathbb{P}\left[H_J(y)\geq\rho(y),\;\forall y\in V\right].
	\]	
\end{prop}				

We are now ready to prove \Cref{thrm:mixing}.
\begin{proof}[Proof of \Cref{thrm:mixing}]
	By \Cref{prop:mixing}, there exists a coupling \(\mathbb{Q}\) of an independent Poisson point process \(\psi\) with intensity measure \(\zeta(y)=\beta(1-\epsilon)\mu_y\mathbbm{1}_{\{y\in\tilde Q_{K'}\}}\) and the locations of the particles \(Y_j\), which are distributed according to the density functions \(f_{\Delta}(x_j,y):=q_{\Delta}(x_j,y)\mu_y\), \(y\in \mathcal{C}_{\infty}\), such that \(\psi\) is a subset of \((Y_j)_{j}\) with probability at least
	\[
		\mathbb{Q}[H_J(y)\geq\beta\mu_y(1-\epsilon),\;\forall y\in \tilde Q_{K'}],
	\]	
	where \(H_J(y)=\sum_{j=1}^J\xi_j f_{\Delta}(x_j,y)\) and \((\xi_j)_{j\leq J}\) are i.i.d.\ exponential random variables with parameter 1. We first observe that the probability of the converse event is
	\begin{eqnarray*}
		\mathbb{Q}[\exists y\in \tilde Q_{K'}:\;H_J(y)<\beta\mu_y(1-\epsilon)]&\leq& \sum_{y\in \tilde Q_{K'}}\mathbb{Q}[H_J(y)<\beta\mu_y(1-\epsilon)]\\
		&\leq&\sum_{y\in \tilde Q_{K'}}e^{\kappa \mu_y\beta(1-\epsilon)}\mathbb{E}^{\mathbb{Q}}[\exp\{-\kappa H_J'(y)\}],
	\end{eqnarray*}
	where we used Markov's inequality in the last step, which is valid for any \(\kappa>0\). Let \(c_1\) be a positive constant which we will fix later and let 
	\[
		R=\sqrt{\Delta}c_1\epsilon^{-1/d}.
	\] 
	Let \(J'\) be a subset of \(\{1,2,\dots,J\}\) such that for each \(\tilde T_i\), \(J'\) contains exactly \(\lceil\sum_{y\in \tilde T_i}\beta\mu_y\rceil\) particles that are inside \(\tilde T_i\).
	Define \(J'(y)\subseteq J'\) to be the set of \(j\in J'\) such that \(|x_j-y|\leq R\) and define \(H'(y)\) as \(H_J(y)\) but with the sum restricted to \(j\in J'(y)\). Since \(H_J(y)\geq H'(y)\) we get that								
	\begin{equation}\label{for:(C)}
		\mathbb{E}^{\mathbb{Q}}[\exp\{-\kappa H_J(y)\}]\leq \mathbb{E}^{\mathbb{Q}}[\exp\{-\kappa H'(y)\}].
	\end{equation}
	Next, we use that the \(\xi_j\) in the definition of \(H\) are independent exponential random variables to obtain
	\begin{eqnarray}\label{for:(B)}
		\mathbb{E}^{\mathbb{Q}}[\exp\{-\kappa H'(y)\}]&=&\prod_{j\in J'(y)}\mathbb{E}^{\mathbb{Q}}[\exp\{-\kappa \xi_jf_{\Delta}(x_j,y)\}]\nonumber\\
		&=&\prod_{j\in J'(y)}\left(1+\kappa f_{\Delta}(x_j,y)\right)^{-1}.
	\end{eqnarray}
											 
	Using Taylor's expansion we have that \(\log(1+x)\geq x-x^2\) for \(|x|\leq\frac{1}{2}\). Since \(\ell\geq S(x)\), we can apply \Cref{prop:heat_kernel_bounds}, to have \(q_{\Delta}(x,y)\leq c_2\Delta^{-d/2}\) for a constant \(c_2>0\) and all \(y\in\tilde Q_{K'}\) and \(x\in J'(y)\).
	Hence if \(\kappa = C\epsilon\Delta^{d/2}\) for the constant \(C=(4C_Uc_2)^{-1}\), then
	\[
		\sup_{x\in B(y,R+\sqrt{d}\ell)}\kappa f_{\Delta}(x,y)=\sup_{x\in B(y,R+\sqrt{d}\ell)}\kappa \mu_y q_{\Delta}(x,y)\leq C_Uc_2\kappa \Delta^{-d/2}<\frac{\epsilon}{4}.
	\]
	For such a value of \(\kappa\) we have
	\begin{eqnarray}\label{for:(A)}
		\prod_{j\in J'(y)}\left(1+\kappa f_{\Delta}(x_j,y)\right)^{-1}&\leq&\prod_{j\in J'(y)}\exp\left\{-\kappa f_{\Delta}(x_j,y)(1-\kappa f_{\Delta}(x_j,y))\right\}\nonumber\\
		&\leq&\exp\left\{-\sum_{j\in J'(y)}\kappa f_{\Delta}(x_j,y)\left(1-\sup_{x\in B(y,R+\sqrt{d}\ell)}\kappa f_{\Delta}(x,y)\right)\right\}\nonumber\\
		&\leq&\exp\left\{-\kappa \sum_{j\in J'(y)} f_{\Delta}(x_j,y)(1-\epsilon/4)\right\}.
	\end{eqnarray}
	We claim that 
	\begin{equation}\label{for:sumfbound}
		\sum_{j\in J'(y)}f_{\Delta}(x_j,y)\geq \beta\mu_y(1-\epsilon/2),
	\end{equation}
	which together with (\ref{for:(A)}), (\ref{for:(B)}) and (\ref{for:(C)}) give that	
	\begin{eqnarray*}
		\mathbb{Q}\left[\exists y\in\tilde Q_{K'}:\;H_J(y)<\beta\mu_y(1-\epsilon)\right] &\leq&\exp\left\{\kappa\mu_y\beta(1-\epsilon)-\kappa \beta\mu_y(1-\epsilon/2)(1-\epsilon/4)\right\}\\
		&\leq&\exp\left\{-\kappa \beta\mu_y\epsilon/4\right\}.
	\end{eqnarray*}
	Using the value of \(\kappa\) gives the theorem.
	
	It remains to show (\ref{for:sumfbound}). For each \(\tilde T_i\) and each particle \(x_j\in \tilde T_i\), let \(x_j'\in \tilde T_i\) be such that \(x_j'=\max_{w\in\tilde T_i}f_{\Delta}(w,y)\). 
	Then, write
	\begin{eqnarray}\label{for:triang_ineq}
		\sum_{j\in J'(y)}f_{\Delta}(x_j,y)&\geq&\sum_{j\in J'(y)}\left(f_{\Delta}(x_{j}',y)-|f_{\Delta}(x_{j}',y)-f_{\Delta}(x_{j},y)|\right).
	\end{eqnarray}										
	
	We have for each \(\tilde T_i\)
	\begin{eqnarray}
		\sum_{\substack{j\in J'(y)\\x_j\in\tilde T_i}}f_{\Delta}(x_j',y)
		&=&\max_{w\in \tilde T_i}f_{\Delta}(w,y)\sum_{\substack{j\in J'(y)\\x_j\in \tilde T_i}}1\nonumber\\
		&\geq&\max_{w\in \tilde T_i}f_{\Delta}(w,y)\sum_{z\in \tilde T_i}\beta\mu_z\nonumber\\
		&\geq&\sum_{z\in \tilde T_i}\beta\mu_z f_{\Delta}(z,y).\label{for:xj_to_xj'}
	\end{eqnarray}	
	Set \(R(y)\) to be the set of all sites \(z\) such that \(|z-y|\leq R-\sqrt{d}\ell\). Note that if \(z\in R(y)\) then for all particles \(x_j\) with \(x'_j=z\) and \(j\in J'\) we have \(j\in J'(y)\). 
	We observe that since \(\mu_z f_{\Delta}(z,y)=\mu_y f_{\Delta}(y,z)\), we have by using (\ref{for:xj_to_xj'}) for each \(\tilde T_i\) that
	\begin{eqnarray*}
		\sum_{j\in J'(y)}f_{\Delta}(x_j',y)&\geq&\sum_{z\in R(y)}\beta\mu_z f_{\Delta}(z,y)\\
		&=&\beta\mu_y\sum_{z\in R(y)}f_{\Delta}(y,z).
	\end{eqnarray*}
	Then, since \(\ell>S^{d+1}(x)\) we have by \Cref{prop:exit} that there exist constants \(c_4\) and \(c_5\) such that 
	\begin{eqnarray}\label{for:fbound1}
		\sum_{j\in J'(y)}f_{\Delta}(x_j',y)&\geq&\beta\mu_y\mathbb{P}_y(\tau(y,R-\sqrt{d}\ell)\geq \Delta)\nonumber\\
		&\geq&\beta\mu_y(1-c_4\exp\{-c_5c_1^2\epsilon^{-2/d}\})\nonumber\\
		&\geq&\beta\mu_y(1-\epsilon/4),
	\end{eqnarray}
	where we set \(c_1\) large enough with respect to \(c_4\) and \(c_5\) for the last inequality to hold.
											
	Now it remains to obtain an upper bound for the term \(\sum_{j\in J'(y)}|f_{\Delta}(x_{j}',y)-f_{\Delta}(x_{j},y)|\).
	We define \(I\) to be the set of all \(i\) such that \(\tilde T_i\) contains a particle \(x_j\) from the set \((x_j)_{j\in J'(y)}\). Then, since \(\ell>S(x)\), there exists positive constants \(C_{PHI}\) and \(C_{BH}\) such that if we apply the PHI (cf. \Cref{prop:PHI}) with	\begin{equation}\label{for:r_0}
		r_0^2=\Delta\geq c_0\ell^2\epsilon^{-4/\Theta}
	\end{equation}
	for some constant \(c_0>d\), we obtain
	
	\begin{eqnarray*}
		\sum_{j\in J'(y)}|f_{\Delta}(x_j',y)-f_{\Delta}(x_j,y)|&=&\sum_{i\in I}\sum_{\substack{j\in J'(y):\\x_j\in \tilde T_i}}|f_{\Delta}(x_j',y)-f_{\Delta}(x_j,y)|\\
		&=&\mu_y\sum_{i\in I}\sum_{\substack{j\in J'(y):\\x_j\in \tilde T_i}}|q_{\Delta}(x_j',y)-q_{\Delta}(x_j,y)|\\
		&\leq&\mu_y\sum_{i\in I}\sum_{\substack{j\in J'(y):\\x_j\in \tilde T_i}}\frac{C_{PHI}\ell^{\Theta}}{\Delta^{\Theta/2}}C_{BH}\Delta^{-d/2}\\
		&\leq&\mu_y\sum_{i\in I}\sum_{x\in \tilde T_i}\frac{2\beta\mu_x C_{PHI}\ell^{\Theta}}{\Delta^{\Theta/2}}C_{BH}\Delta^{-d/2},
	\end{eqnarray*}
	where in the first inequality we replaced the supremum term coming from \Cref{prop:PHI} by its upper bound \(C_{BH}\Delta^{-d/2}\) from \Cref{prop:heat_kernel_bounds}, and used that \(r_0=\sqrt{\Delta}\) in the bound from \Cref{prop:PHI}. Then
	\begin{eqnarray}\label{for:fbound2}
		\sum_{j\in J'(t)}|f_{\Delta}(x_j',y)-f_{\Delta}(x_j,y)|&\leq&2\beta\mu_y C_{PHI}C_{BH}\sum_{i\in I}\sum_{x\in \tilde T_i}\mu_x\ell^{\Theta}\Delta^{-(d+\Theta)/2}\nonumber\\
		&\leq&2\beta\mu_y C_{PHI}C_{BH}C_UR^d\ell^{\Theta}\Delta^{-(d+\Theta)/2}\nonumber\\
		&\leq&\beta\mu_y\frac{\epsilon}{4},
	\end{eqnarray}
	where the last inequality holds by using \(\Delta\geq c_0\ell^2\epsilon^{-4/\Theta}\) and setting \(c_0>(2C_{PHI}C_{BH}C_Uc_1^d)^{-2/\theta}\). Note that in order to use \Cref{prop:PHI}, we need to have that each pair \(x_j,x_j'\) is contained in some ball \(B(x_0,r_0/2)\). This is satisfied since \(\|x_j-x_j'\|\leq\sqrt{d}\ell\) and \(r_0\) is set sufficiently large by (\ref{for:r_0}).
	Plugging (\ref{for:fbound2}) and (\ref{for:fbound1}) into (\ref{for:triang_ineq}) proves (\ref{for:sumfbound}).
	\end{proof}

\section{Extensions}\label{sect:extensions}
				
Although the estimate derived in \Cref{thrm:mixing} does not depend on the particles outside of \(Q_K\) at time \(0\) when \(K-K'\) is sufficiently large, it still depends on the geometry of the entire graph outside of \(Q_K\). In some applications, as in our companion paper \cite{Gracar2016a}, one needs to apply this coupling in many different regions of the graph simultaneously. In such cases, in order to control dependences between different regions, it is important that the coupling procedure depends only on the local structure of the graph. In order to do this, we will condition the particles to be inside some large enough, but finite region while they move for time \(\Delta\). Recall that, for any \(\rho>0\), \(Q_{\rho}=[-\rho/2,\rho/2]^d\) is the cube of side length \(\rho\). For any \(\rho>0\), we say that a random walk has \emph{displacement} in \(Q_{\rho}\) during \([0,\Delta]\) if the random walk never exits \(x+Q_{\rho}\) during the time interval \([0,\Delta]\), where \(x\) is the starting vertex of the random walk.
				
\begin{lemma}\label{lem:gtoq}
	Let \(\mu_{x,y}\) satisfy Assumptions \ref{ass:S} and \ref{ass:critical}. There exist constants \(c_1\) and \(c_2\) so that the following holds. Let \(V=\mathbb{Z}^d\), \(\ell>0\) and consider the cube \(Q_{\ell}\). Assume \(\ell>S(x)\) for all \(x\in Q_{\ell}\). Let \(\Delta>c_1\ell^2\) and \(\rho\geq c_2\sqrt{\Delta\log\Delta}\). Consider a random walk \(Y\) that moves along \(\mathcal{G}\) for time \(\Delta\) conditioned on having its displacement in \(Q_{\rho}\) during the time interval \([0,\Delta]\). Let \(x,y\in Q_\ell\) with \(x\) being the starting point of the walk, and define 
	\[g(x,y):=\mathbb{P}_x\left[Y_{\Delta}=y\,|\,Y\textrm{ has displacement in }Q_{\rho}\textrm{ during }[0,\Delta]\right].
	\]
	Then there exists a constant \(C>2\) such that for \(x,y,z\in Q_{\ell}\) we have
	\[
		\left|\frac{g(x,y)}{\mu_y}-\frac{g(z,y)}{\mu_y}\right|\leq C\ell^{\Theta}\Delta^{-(d+\Theta)/2}.
	\]
\end{lemma}

\begin{remark}
	Note that the above bound has the same form as the one for the heat kernel of unconditioned random walks in \Cref{prop:PHI}, with the supremum being bounded above by the heat kernel bound from \Cref{prop:heat_kernel_bounds}. This allows us to extend \Cref{thrm:mixing} to random walks conditioned to have a bounded displacement during \([0,\Delta]\).
\end{remark}

\begin{proof}[Proof of \Cref{lem:gtoq}]
	Denote by \(p_E(\rho)\) the probability that a random walk started at \(x\) has displacement in \(Q_{\rho}\) during \([0,\Delta]\). From \Cref{prop:exit} , we have that if \(\Delta\) is sufficiently big, then
	\begin{align}
		1-p_{E}(\rho)&\leq \mathbb{P}_x[Y \textrm{ exits }B(x,\rho/2)\textrm{ during }[0,\Delta]]\nonumber\\
		&=\mathbb{P}_x(\tau(x,\rho/2)<\Delta)\nonumber\\
		&\leq c_a\exp\{-c_b\rho^2/\Delta\}.\label{for:displacement_bound}
	\end{align}
											
	Next, using \(h(x,y):=\mathbb{P}_x\left[Y_{\Delta}=y\,|\,Y\textrm{ exits }x+Q_{\rho}\textrm{ during }[0,\Delta]\right]\) and \(f_{\Delta}(x,y)=\mathbb{P}_x[Y_{\Delta}=y]\), we can write
	\[
		f_{\Delta}(x,y)=g(x,y)p_E(\rho)+h(x,y)(1-p_E(\rho)).
	\]
	With this we have
	\begin{equation}\label{for:exittime}
		g(x,y)\leq f_{\Delta}(x,y)\frac{1}{p_E(\rho)}.
	\end{equation}
											
	Then, we can write
	\begin{eqnarray*}
		\left|\frac{g(x,y)}{\mu_y}-\frac{g(z,y)}{\mu_y}\right|&=&\mathds{1}_{\{g(x,y)>g(z,y)\}}\left(\frac{g(x,y)}{\mu_y}-\frac{g(z,y)}{\mu_y}\right)\\
		&&+\mathds{1}_{\{g(x,y)<g(z,y)\}}\left(\frac{g(z,y)}{\mu_y}-\frac{g(x,y)}{\mu_y}\right)\\
		&\leq&\mathds{1}_{\{g(x,y)>g(z,y)\}}\left(\frac{f_{\Delta}(x,y)}{\mu_yp_E(\rho)}-\frac{f_{\Delta}(z,y)}{\mu_yp_E(\rho)}+\frac{h(z,y)(1-p_E(\rho))}{p_E(\rho)\mu_y}\right)\\
		&&+\mathds{1}_{\{g(x,y)<g(z,y)\}}\left(\frac{f_{\Delta}(z,y)}{\mu_yp_E(\rho)}-\frac{f_{\Delta}(x,y)}{\mu_yp_E(\rho)}+\frac{h(x,y)(1-p_E(\rho))}{p_E(\rho)\mu_y}\right)\\
		&\leq&\frac{|q_{\Delta}(y,x)-q_{\Delta}(y,z)|}{p_E(\rho)}+\frac{\max\{h(x,y),h(z,y)\}(1-p_E(\rho))}{p_E(\rho)\mu_y}.
	\end{eqnarray*}
	
	Note that \(h(x,y)\) can be written as \(f_{\Delta-\tau}(w,y)\), where \(\tau\) is the first time \(Y\) exists \(x+Q_{\rho}\) and \(w\) is the random vertex at the boundary of \(x+Q_{\rho}\) where \(Y\) is at time \(\tau\). Since the weights \(\mu_{x,y}\) satisfy (\ref{for:graphupperbound}) by \Cref{ass:S}, we have that \(\frac{f_{\Delta-\tau}(w,y)}{\mu_y}\) is at most some positive constant \(c\). This holds because either \(\Delta-\tau\) is larger than \(|w-y|\), which allows us to apply heat kernel bounds from \Cref{prop:heat_kernel_bounds}, or \(\Delta-\tau\) is smaller than \(|w-y|\) so \(f_{\Delta-\tau}(w,y)\) is bounded above by the probability that a random walk jumps at least \(|w-y|\) steps in time \(\Delta-\tau\), which is small enough since \(|w-y|\) is large. This gives that \(\frac{\max\{h(x,y),h(z,y)\}}{\mu_y}\) is at most \(c\). With this and (\ref{for:displacement_bound}) we obtain that
	\begin{align*}
		\frac{\max\{h(x,y),h(z,y)\}(1-p_E(\rho))}{\mu_y p_E(\rho)}&\leq \frac{c c_a}{p_E(\rho)}\exp\left\{\frac{-c_b\rho^2}{\Delta}\right\}\\
		&\leq\frac{c c_a}{p_E(\rho)}\exp\left\{-c_b c_2\log\Delta\right\}.
	\end{align*}
	By (\ref{for:displacement_bound}) we can just bound \(p_E(\rho)\geq 1/2\) above. Then,  applying \Cref{prop:A} to \(|q_{\Delta}(y,x)-q_{\Delta}(y,z)|\), and using \Cref{prop:heat_kernel_bounds} to bound the resulting supremum term, concludes the proof.
\end{proof}

The next theorem is an adaptation of \Cref{thrm:mixing} for conditioned random walks. Note that we need a stronger condition on \(K-K'\) below than in \Cref{thrm:mixing}.
\begin{thrm}\label{thrm:mixing2}
	Let \(\mu_{x,y}\) satisfy Assumptions \ref{ass:S} and \ref{ass:critical}. There exist constants \(c_0\), \(c_1\), \(C>0\) such that the following holds.
	Fix \(K>\ell>0\) and \(\epsilon\in(0,1)\). Consider the cube \(Q_K\) tessellated into subcubes \((T_i)_{i}\) of side length \(\ell\) and assume that \(\ell>S^{d+1}(x)\) for all \(x\in \tilde Q_K\). 
	Let \((x_j)_{j}\subset \tilde Q_K\) be the locations at time \(0\) of a collection of particles, such that each subcube \(\tilde T_i\) contains at least \(\sum_{y\in \tilde T_i}\beta\mu_y\) particles for some \(\beta>0\).
	Let \(\Delta\geq c_0\ell^2\epsilon^{-4/\Theta}\), where \(\Theta\) is as in \Cref{prop:PHI}.
	Fix \(K'>0\) such that \(K-K'\geq c_1\sqrt{\Delta\log\Delta}\). 
	For each \(j\), denote by \(Y_j\) the location of the \(j\)-th particle at time \(\Delta\), conditioned on having displacement in \(Q_{K-K'}\) during \([0,\Delta]\). 
	Then there exists a coupling \(\mathbb{Q}\) of an independent Poisson point process \(\psi\) with intensity measure \(\zeta(y)=\beta(1-\epsilon)\mu_y\), \(y\in \tilde Q_K\), and \((Y_j)_{j}\) such that within \(\tilde Q_{K'}\subset \tilde Q_K\), \(\psi\) is a subset of \((Y_j)_{j}\) with probability at least
	\[
		1-\sum_{y\in \tilde Q_{K'}}\exp\left\{-C\beta\mu_y\epsilon^2\Delta^{d/2}\right\}.
	\]
\end{thrm}
				
\begin{proof}
	Using \Cref{lem:gtoq} and (\ref{for:exittime}) when setting \(\kappa\), the proof goes in the same way as the proof of \Cref{thrm:mixing}. The independence from \(G\) outside of \(\tilde Q_{K}\) follows from the fact that we only consider particles which have displacement in \(Q_{K-K'}\) and ended in \(\tilde Q_{K'}\), so that they never left \(\tilde Q_{K}\) during \([0,\Delta]\). 
\end{proof}
		
\subsection{Extension to other graphs}\label{rem:general}
		
We have shown that the local mixing result of Theorems \ref{thrm:mixing} and \ref{thrm:mixing2} work for \(\mathbb{Z}^d\), but they can easily be extended to the more general graphs defined in \Cref{sect:setup}, as long as Assumptions \ref{ass:S} and \ref{ass:critical} hold.

We start with a region \(A\subseteq \mathcal{C}_{\infty}\) around the origin of \(G\) and tesselate it into tiles \((T_i)_{i\in I}\) of diameter at most \(\ell\). Let \(\Delta\) be as in \Cref{thrm:mixing}. Let \(A'\subset A\) be all the sites in \(A\) that are at least \(\sqrt{\Delta}c_1\epsilon^{-1/d}+c\ell\) away from the boundary of \(A\). Then, if \(A'\) is not empty, using the same steps as in the proof of \Cref{thrm:mixing}, if each tile \(T_i\) of \(A\) contains at least \(\beta\sum_{y\in T_i}\mu_y\) particles at time \(0\), it holds that in the region \(A'\), there is a coupling with an independent Poisson point process \(\psi\) of intensity measure \(\zeta(y)=\beta(1-\epsilon)\mu_y\) such that at time \(\Delta\) the particles inside \(A'\) are contained in \(\psi\) with probability at least
\[
	1-\sum_{y\in A'}\exp\left\{-C\beta\mu_y\epsilon^2\Delta^{d/2}\right\},
\]
for some constant \(C>0\).

Furthermore, \Cref{thrm:mixing2} can analogously be extended in the same way, if we require that \(A'\) contains only sites that are at least \(c_1\sqrt{\Delta\log\Delta}\) away from the boundary of \(A\), for some constant \(c_1\), and if we condition the random walks to have their displacement limited to a ball of radius \(c_1\sqrt{\Delta\log\Delta}\).

\section{Spread of the infection}\label{sect:spread}

Our goal in this section will be to use \Cref{thrm:mixing2} in order to show that on the graph \(G=(V,E)\) with \(V=\mathbb{Z}^d\) and \(E=\{(x,y):\|x-y\|_1=1\}\), and with \(\mu_{x,y}\), \((x,y)\in E\) being i.i.d.\ and satisfying (\ref{eq:mu_bounds_new}), information spreads with positive speed in any direction, as claimed in Theorems \ref{thrm:spread} and \ref{thrm:spread_recov}. In this setting, \Cref{prop:A} guarantees that \Cref{ass:S} holds with \(S(x)\equiv 1\) and since \(\mu_{x,y}\neq 0\) for all \((x,y)\in E\), we also have that \Cref{ass:critical} holds. 

Recall that we assume \(d\geq 2\). Tessellate \(\mathbb{Z}^d\) into cubes of side length \(\ell\), indexed by \(i\in \mathbb{Z}^d\). Next, tessellate time into intervals of length \(\beta\), indexed by \(\tau\in\mathbb{Z}\). With this we denote by the space-time cell \((i,\tau)\in\mathbb{Z}^{d+1}\) the region \(\prod_{j=1}^d[i_j\ell,(i_j+1)\ell]\times[\tau\beta,(\tau+1)\beta]\). In the following, \(\beta\) is set as a function of \(\ell\) so that the ratio \(\beta/\ell^2\) is fixed first to be a small constant, and then \(\ell\) is set sufficiently large.

We will use a result from \cite{Gracar2016a} that gives the existence of a Lipschitz connected surface (cf. Definitions \ref{def:lip_fun} and \ref{def:lip_surf} below) that surrounds the origin and which is composed of space-time cells, for which a certain local event holds. This will allow us to obtain an infinite sequence of space-time cells, such that the infection spreads from one cell to the next.

In order to obtain this result, we will need to consider overlapping space-time cells. Let \(\eta\geq 1\) be an integer which will represent the amount of overlap between cells. For each cube \(i=(i_1,\dots,i_d)\) and time interval \(\tau\), define the \emph{super cube} \(i\) as \(\prod_{j=1}^d[(i_j-\eta)\ell,(i_j+\eta+1)\ell]\) and the \emph{super interval} \(\tau\) as \([\tau\beta,(\tau+\eta)\beta]\). We define the \emph{super cell} \((i,\tau)\) as the Cartesian product of the super cube \(i\) and the super interval \(\tau\).

In the following we will say a particle has displacement inside \(X'\) during a time interval \([t_0,t_0+t_1]\), if the location of the particle at all times during \([t_0,t_0+t_1]\) is inside \(x+X'\), where \(x\) is the location of the particle at time \(t_0\). For each time $s\geq 0$, let $\Pi_s$ be a point process on V, which represents the locations of the particles at time $s$. We say an event \(E\) is \emph{increasing} for \((\Pi_s)_{s\geq 0}\) if the fact that \(E\) holds for \((\Pi_s)_{s\geq 0}\) implies that it holds for all \((\Pi'_s)_{s\geq 0}\) for which \(\Pi_s'\supseteq \Pi_s\) for all \(s\geq 0\). We say an event \(E\) is \emph{restricted} to a region \(X\subset\mathbb{Z}^d\) and a time interval \([t_0,t_1]\) if it is measurable with respect to the \(\sigma\)-field generated by all the particles that are inside \(X\) at time \(t_0\) and their positions during \([t_0,t_1]\). For an increasing event \(E\) that is restricted to a region \(X\) and time interval \([t_0,t_1]\), we have the following definition.

\begin{mydef}\label{def:probassoc}
	\(\nu_E\) is called the \emph{probability associated} to a an increasing event \(E\) that is restricted to \(X\)  and a time interval \([t_0, t_0+t_1]\) if, for an intensity measure \(\zeta\), \(\nu_E(\zeta,X,X',t_1)\) is the probability that \(E\) happens given that, at time \(t_0\), the particles in \(X\) are given by a Poisson point process of intensity measure \(\zeta\) and their motions from \(t_0\) to \(t_0+t_1\) are independent continuous time random walks on the weighted graph \((G,\mu)\), where the particles are conditioned to have displacement inside \(X'\).
\end{mydef}

For each \((i,\tau)\in\mathbb{Z}^{d+1}\), let \(E_{\mathrm{st}}(i,\tau)\) be an increasing event restricted to the super cube \(i\) and the super interval \(\tau\). Here the subscript \(\mathrm{st}\) refers to space-time. We say that a cell \((i,\tau)\) is \emph{bad} if \(E_{\mathrm{st}}(i,\tau)\) does not hold and \emph{good} otherwise.

We will need a different way to index space-time cells, which we refer to as the \emph{base-height index}. In the base-height index, we pick one of the \(d\) spatial dimensions and denote it as \emph{height}, using index \(h\in\mathbb{Z}\), while the remaining \(d\) space-time dimensions form the base, which we index by \(b\in\mathbb{Z}^d\). In this way, for each space-time cell \((i,\tau)\) there will be \((b,h)\in\mathbb{Z}^{d+1}\) such that the base-height cell \((b,h)\) corresponds to the space-time cell \((i,\tau)\)

We analogously define the \emph{base-height super cell} \((b,h)\) to be the space-time super cell \((i,\tau)\), for which the base-height cell \((b,h)\) corresponds to the space-time cell \((i,\tau)\). Similarly, we define \(E_\mathrm{bh}(b,h)\), the increasing event restricted to the super cell \((b,h)\) that is the same as the event \(E_{\mathrm{st}}(i,\tau)\) for the space-time super cell \((i,\tau)\) that corresponds to the base-height super cell \((b,h)\). Here, the subscript \(\mathrm{bh}\) refers to the base-height index.

In order to prove Theorems \ref{thrm:spread} and \ref{thrm:spread_recov}, we will need a theorem from \cite{Gracar2016a}, which gives the existence of a two-sided Lipschitz surface \(F\).

\begin{mydef}\label{def:lip_fun}
	A function \(F:\mathbb{Z}^d\rightarrow \mathbb{Z}\) is called a \emph{Lipschitz function}
	 if \(|F(x)-F(y)|\leq 1\) whenever \(\|x-y\|_1 = 1\).
\end{mydef}

\begin{mydef}\label{def:lip_surf}
	A \emph{two-sided Lipschitz surface} \(F\) is a set of base-height cells \((b,h)\in\mathbb{Z}^{d+1}\) such that for all \(b\in\mathbb{Z}^d\) there are exactly two (possibly equal) integer values \(F_+(b)\geq 0\) and \(F_-(b)\leq0\) for which \((b,F_+(b)),(b,F_-(b))\in F\) and, moreover, \(F_+\) and \(F_-\) are Lipschitz functions.
\end{mydef}

We say a space-time cell \((i,\tau)\) belongs to \(F\) if there exists a base-height cell \((b,h)\in F\) that corresponds to \((i,\tau)\). We say a two-sided Lipschitz surface \(F\) is finite, if for all $b\in\mathbb{Z}^d$, we have $F_+(b)<\infty$ and $F_-(b)>-\infty$. For a positive integer $D$, we say a two-sided Lipschitz surface \emph{surrounds} a cell \((b',h')\) at distance \(D\) if any path \((b',h')=(b_0,h_0),(b_1,h_1),\dots,(b_n,h_n)\) for which \(\|(b_i,h_i)-(b_{i-1},h_{i-1})\|_1=1\) for all \(i\in\{1,\dots n\}\) and \(\|(b_n,h_n)-(b_0,h_0)\|_1>D\), intersects with \(F\).

We now present the main result from our paper \cite{Gracar2016a}, which holds for graphs where a local mixing result, such as the one in \Cref{thrm:mixing2}, hold. More precisely, for a graph satisfying Assumption \ref{ass:S} and (\ref{eq:mu_bounds_new}) (which implies \Cref{ass:critical} holds) we have that \Cref{thrm:mixing2} holds (with \(S(x)=1\) for all \(x\in V\)), which in turn gives that the following result from \cite{Gracar2016a} holds. Recall that, for any \(\rho\geq 2\), \(Q_{\rho}\) stands for the cube \([-\rho/2,\rho/2]^d\), and  that \(\lambda\) is the intensity measure of the Poisson point process of particles as defined in \Cref{sect:intro}.


\begin{thrm}\label{thrm:exp_tail}
	Let \(\mathcal{G}=(G,\mu)\) be a graph satisfying Assumption \ref{ass:S} and (\ref{eq:mu_bounds_new}) on the lattice \(\mathbb{Z}^d\) for \(d\geq 2\). There exist positive constants \(c_1\) and \(c_2\) such that the following holds. Tessellate \(G\) in space-time cells and super cells as described above for some \(\ell,\beta,\eta>0\) such that the ratio \(\beta/\ell^2\) is small enough. 
	Let \(E_{\mathrm{st}}(i,\tau)\) be an increasing event, restricted to the space-time super cell \((i,\tau)\). 
	Fix \(\epsilon\in(0,1)\) and fix \(w\) such that
	\[
		w\geq\sqrt{\frac{\eta\beta}{c_2\ell^2}\log\left(\frac{8c_1}{\epsilon}\right)}.
	\] 
	Then, there exists a positive number \(\alpha_0\) that depends on \(\epsilon\), \(\eta\) and that ratio \(\beta/\ell^2\) so that if
	\[
	\min\left\{C_{M}^{-1}\epsilon^2\lambda_0\ell^d,\log\left(\frac{1}{1-\nu_{E_{\mathrm{st}}}((1-\epsilon)\lambda,Q_{(2\eta+1)\ell},Q_{w\ell},\eta\beta)}\right)\right\}\geq\alpha_0,
	\]
	a two-sided Lipschitz surface \(F\) where \(E_{\mathrm{st}}(i,\tau)\) holds for all \((i,\tau)\in F\) exists. Furthermore, the surface is finite almost surely and surrounds the origin at a finite distance almost surely.
\end{thrm}

Recall that we want to show that the infection spreads with positive speed. Given a space-time tessellation of \(G\) and a local increasing event \(E_{\mathrm{st}}\), \Cref{thrm:exp_tail} gives the existence of a Lipschitz surface \(F\) on which \(E_{\mathrm{st}}\) holds. Let \(T=\ell^{5/3}\). We will define the increasing event \(E_{\mathrm{st}}(i,\tau)\) to represent a single infected particle in the middle of the super cube \(i\) at time \(\tau\beta\) infecting a large number of particles in that super cube by time \(\tau\beta+T\), after which the infected particles move up to time $(\tau+1)\beta$, spreading to all of the cubes contained in the super cube.

Let \((i,\tau)\) be a space-time cell as defined previously. We consider that there is an infected particle in the center cube of the super cube \(i\) at time \(\tau\beta\), that is,  the particle is inside \(\prod_{j=1}^d[i_j\ell,(i_j+1)\ell]\). Starting from time \(\tau\beta\), we let the infected particle move and infect sufficiently many other particles by time \(\tau\beta+T\). This is given in the lemma below.

\begin{lemma}\label{lemma:particles2}
	There exist positive constant \(C_1\) such that the following holds for all large enough \(\ell\).
	Let \(Q^*=\prod_{j=1}^d[(i_j-\eta)\ell,(i_j+\eta+1)\ell]\) and let \(\left(\rho(t)\right)_{\tau\beta\leq t\leq \tau\beta+T}\) be the path of an infected particle that starts in \(\prod_{j=1}^d[i_j\ell,(i_j+1)\ell]\) and stays inside \(\prod_{j=1}^d[(i_j-\eta+1)\ell,(i_j+\eta)\ell]\) during \([\tau\beta,\tau\beta+T]\). Assume that at time \(\tau\beta\), the number of particles at each vertex \(x\in Q^*\setminus\rho(\tau\beta)\) is a Poisson random variable of mean \(\frac{\lambda_0}{2}\mu_x\).
	Let $\Upsilon$ be the subset of those particles that do not leave $Q^*$ during $[\tau\beta,\tau\beta+T]$, and let $\Upsilon'\subset \Upsilon$ be the particles colliding with the path $\rho$, that is, for each particle of $\Upsilon'$ there exists a time $t\in[\tau\beta,\tau\beta+T]$ such that the particle is located at $\rho(t)$. Then, $\Upsilon'$ is a Poisson random variable of mean at least
	\(
		C_1\lambda_0\ell^{1/3}.
	\)
\end{lemma}

\begin{proof}
	For each time $t\in [\tau\beta,\tau\beta+T]$, let $\Psi_t$ be the Poisson point process on $V$ giving the locations at time $t$ of the particles that belong to $\Upsilon$. Since the particles that start in \(Q^*\) move around and can leave \(Q^*\), we need to find a lower bound for the intensity of \(\Psi_t\) for times in \([\tau\beta,\tau\beta+T]\). Note that the infected particle we are tracking is not part of \(\Psi\), since \(\Psi\) does not include particles located at \(\rho(\tau\beta)\) at time \(\tau\beta\). 
	
	We will need to apply heat kernel bounds from \Cref{prop:heat_kernel_bounds} to the particles in \(Q^*\), so we need to ensure that the time intervals we consider are large enough for the proposition to hold. 
	We will only consider times \(t\in[\ell^{4/3},T]\) so that for large enough \(\ell\), \(t\geq\sup_{\substack{x\in Q^*\\y\in Q^*}}\|x-y\|_1\) and so the heat kernel bounds from \Cref{prop:heat_kernel_bounds} hold. Then, we have that for all sites \(x\in Q^*\) that are at least \(\ell\) away from the boundary of \(Q^*\) and at any such time \(t\) the intensity of \(\Psi_{\tau\beta+t}\) at vertex $x\in V$ is at least
	\begin{equation*}
		\psi(x,\tau\beta+t)\geq\sum_{\substack{y\in Q^*\\y\neq \rho(\tau\beta)}}\tfrac{\lambda_0}{2}\mu_y\cdot\mu_xq_t(y,x)
		= \tfrac{\lambda_0}{2}\mu_x\sum_{\substack{y\in Q^*\\y\neq \rho(\tau\beta)}}\mathbb{P}_x[Y_t=y],
	\end{equation*}
	where we used in the last step that the heat kernel \(q_t\) is symmetric. We now use the exit time bound from \Cref{prop:exit}  to get that 
	\[
		\sum_{\substack{y\in Q^*}}\mathbb{P}_x[Y_t=y]\geq 1-c_3\exp\{-c_4\ell^2/t\}.
	\]
	Next, we use that \(\mathbb{P}_x[Y_t=y]=\mu_y q_t(x,y)\leq C_M q_t(x,y)\), and use \Cref{prop:heat_kernel_bounds} to account for the particles at \(\rho(\tau\beta)\), yielding 
	\begin{equation*}
		\sum_{\substack{y\in Q^*\\y\neq \rho(\tau\beta)}}\mathbb{P}_x[Y_t=y]\geq 1-c_3\exp\left\{-c_4\ell^2/t\right\}-C_Mc_5t^{-d/2}.
	\end{equation*}
	This gives that for any \(t\in[\ell^{4/3},T]\), the intensity of \(\Psi_{\tau\beta+t}\) is at least
	\[
		\psi(x,\tau\beta+t)\geq \tfrac{\lambda_0}{2} \mu_x(1-c_3\exp\{-c_4\ell^2/T\}-C_Mc_5\ell^{-2d/3}).
	\]

	Let \([\tau\beta,\tau\beta+T]\) be divided into subintervals of length \(W\in(0,T]\), where we set \(W=\ell^{4/3}\) so that it is large enough to allow the use of the heat kernel bounds from \Cref{prop:heat_kernel_bounds}. Let \(J=\{1,\dots, \lfloor T/W\rfloor \}\) and \(t_j:=\tau\beta+jW\). Then the intensity of particles that share a site with the initially infected particle only at one time among \(\{t_1, t_2, \dots, t_{ \lfloor T/W\rfloor }\}\) is at least
	\begin{align*}
		&\sum_{j\in J}\psi(\rho(t_j),t_j)\mathbb{P}_{\rho(t_j)}[X_{r-t_j}\neq \rho(r) \;\forall r\in\{t_{j+1},\dots,t_{\lfloor T/W\rfloor}\}]\\
		&\geq \tfrac{\lambda_0}{2} C_M^{-1}(1-c_3\exp\{-c_4\ell^2/T\}-C_Mc_5\ell^{-2d/3})\sum_{j\in J}\left(1-\sum_{z>j}\mathbb{P}_{\rho(t_j)}[X_{t_z-t_j}= \rho(t_z)]\right).
	\end{align*}

	We want to make all of the terms of the sum over \(J\) positive, so we consider the term \(\sum_{z>j}\mathbb{P}_{\rho(t_j)}[X_{t_z-t_j}= \rho(t_z)]\) and show that it is smaller than \(\frac{1}{2}\) for large enough \(\ell\). To do this, we use that \(\mathbb{P}_x[Y_t=y]=\mu_y q_t(x,y)\) with the heat kernel bounds from \Cref{prop:heat_kernel_bounds}, which hold when \(W\geq\ell^{4/3}\) and \(\ell\) is large enough, to bound it from above by
	\begin{align}
		\sum_{z>j}\mathbb{P}_{\rho(t_j)}[X_{t_z-t_j}= \rho(t_z)]
		&\leq \sum_{z>j}C_MC_{HK}(t_z-t_j)^{-d/2}\nonumber\\
		&\leq C_M C_{HK}W^{-d/2}\sum_{z=1}^{T/W-j}z^{-d/2}\label{for:TW}
	\end{align}
	where \(C_{HK}\) is the constant coming from \Cref{prop:heat_kernel_bounds}. Then, (\ref{for:TW}) can be bound from above by
	\begin{equation}
		C_M C_{HK} W^{-d/2}\left(2+\sum_{z=3}^{T/W-j}z^{-d/2}\right)
		\leq C_M C_{HK}W^{-d/2}\left(2+\int_{2}^{T/W}z^{-d/2}dz\right)\label{for:TWintegral}.
	\end{equation}
	Let \(C\) be a constant that can depend on \(C_{HK}\), \(C_M\) and \(d\). Then for \(d=2\), (\ref{for:TWintegral}) it is smaller than \(CW^{-1}\log(T/W)\), and for \(d\geq 3\) the expression in (\ref{for:TWintegral}) is smaller than \(CW^{-d/2}\). Thus, setting \(\ell\) large enough, both terms are smaller than \(\frac{1}{2}\).

	Then, as a sum of Poisson random variables, we get that \(\Upsilon'\) is a Poisson random variable with a mean at least
	\[
		\tfrac{\lambda_0}{2} C_M^{-1}(1-c_3\exp\{-2c_4\ell^2/T\}-C_Mc_5\ell^{-2d/3})\tfrac{T}{2W}.
	\]
	Using that \(T=\ell^{5/3}\) and setting \(\ell\) large enough establishes the lemma, with \(C_1\) being any constant satisfying \(C_1<\frac{C_M^{-1}}{4}\).
\end{proof}

Next we show that the particles from \Cref{lemma:particles2} move to nearby cells, spreading the infection.

\begin{lemma}\label{lemma:particles1}
	Let \(z=(z_1,\dots,z_d)\) with \(z_j\in\{-\eta,-\eta+1,\dots,\eta\}\) for all \(j\in\{1,\dots d\}\), and fix the ratio \(\beta/\ell^2\). Let \(A(i,\tau,N,z)\) be the event that given a set of \(N>0\) particles in \(\prod_{j=1}^d[(i_j-\eta)\ell,(i_j+\eta+1)\ell]\) at time \(\tau\beta+T\),
	at least one of them is in \(\prod_{j=1}^d[(i_j+z_j)\ell,(i_j+z_j+1)\ell]\) at time \((\tau+1)\beta\). Then, if \(\ell\) is sufficiently large while keeping \(\beta/\ell^2\) fixed, we obtain
	\[
	\mathbb{P}[A(i,\tau,N,z)]\geq 1-\exp\{-Nc_p\},
	\]
	where \(c_p\) is a positive constant that is bounded away from \(0\) and depends only on \(d\), \(\eta\) and the ratio \(\beta/\ell^2\).
\end{lemma}

\begin{proof}
	Let \(Q^*=\prod_{j=1}^d[(i_j-\eta)\ell,(i_j+\eta+1)\ell]\) and \(Q^{**}=\prod_{j=1}^d[(i_j+z_j)\ell,(i_j+z_j+1)\ell]\). For \(t^{2/3}\geq\sup_{\substack{x\in Q^*\\y\in Q^{**}}}\|x-y\|_1\), define \(p_t:=\inf_{\substack{x\in Q^*}}\sum_{y\in Q^{**}}\mathbb{P}_x[Y_t=y]\). Then, if we define \(\mathrm{bin}(N,p_t)\) to be a binomial random variable of parameters \(N\in\mathbb{N}\) and \(p_t\in[0,1]\), it directly follows that
	\[
		\mathbb{P}[A(i,\tau,N,z)]\geq\mathbb{P}[\mathrm{bin}(N,p_t)\geq1]\geq1-\exp\{-Np_t\}.
	\]

	It remains to show that for \(t=\beta-T\), we have that \(p_t\geq c_p>0\) for some constant \(c_p\). We will use the heat kernel bounds for the pair \(x,y\), which hold if \(\|x-y\|_1^{3/2}\leq\beta-T\) for all \(x\in Q^*,y\in Q^{**}\). Given the ratio \(\beta/\ell^2\), \(d\) and \(\eta\), this is satisfied if \(\ell\) is large enough. Then we have that
	\begin{align*}
		p_{\beta-T}&=\inf_{x\in Q^*}\sum_{y\in Q^{**}}\mathbb{P}_x[Y_{\beta-T}=y]\\
		&\geq\inf_{x\in Q^*}C_{M}^{-1}\sum_{y\in Q^{**}}q_{{\beta-T}}(x,y)\\
		&\geq\inf_{x\in Q^*}C_{M}^{-1}\sum_{y\in Q^{**}}c_1 \beta^{-d/2}\exp\left\{-c_2\frac{\|x-y\|_1^2}{{\beta-T}}\right\}.
	\end{align*}
	Now we use that \(x\) and \(y\) can be at most \(c_\eta\ell\) apart where \(c_\eta\) is a constant depending on \(d\) and \(\eta\) only, and that \(\beta-T\geq\beta/2\) for \(\ell\) large enough. Hence,
	\begin{align*}
		p_{{\beta-T}}&\geq \inf_{x\in Q^*}C_M^{-1}\sum_{y\in Q^{**}}c_1 \beta^{-d/2}\exp\left\{-c_2\frac{2(c_\eta\ell)^2}{\beta}\right\}\\
		&= C_M^{-1}c_1\ell^d\left(\frac{1}{\beta}\right)^{d/2}\exp\left\{-c_2\frac{2(c_\eta\ell)^2}{\beta}\right\}\\
		&\geq c_p.
	\end{align*}
\end{proof}

In the next lemma, we will tie together the results from \Cref{lemma:particles2} and \Cref{lemma:particles1}. In order to precisely describe the behavior of the particles involved, we say a particle \(x\) \emph{collides} with particle \(y\) during a time interval \([t_0,t_1]\), if for at least one \(t\in[t_0,t_1]\), \(x\) and \(y\) are at the same site.

\begin{lemma}\label{prop:event}
	Consider the super cell \((i,\tau)\). Assume that at each site \(x\in \prod_{j=1}^d[(i_j-\eta)\ell,(i_j+\eta+1)\ell]\) the number of particles at \(x\) at time \(\tau\beta\) is a Poisson random variable of intensity \(\frac{\lambda_0}{2}\mu_x\), and let \(\Upsilon\) be the collection of such particles. Assume that, at time \(\tau\beta\), there is at least one infected particle \(x_0\) inside \(\prod_{j=1}^d[i_j\ell,(i_j+1)\ell]\). Let \(E_{\mathrm{st}}(i,\tau)\) be the event that at time \((\tau+1)\beta\), for all \(i'\in\mathbb{Z}^d\) with \(\|i-i'\|_{\infty}\leq\eta\), there is at least one particle  from \(\Upsilon\) in \(\prod_{j=1}^d[(i_j')\ell,(i_j'+1)\ell]\) that collided with \(x_0\) during \([\tau\beta,\tau\beta+T]\). If \(\ell\) is sufficiently large for Lemmas \ref{lemma:particles2} and \ref{lemma:particles1} to hold,
	then there exists a positive constant \(C\) such that
	\[
		\mathbb{P}[E_{\mathrm{st}}(i,\tau)]\geq1-\exp\{-C\lambda_0\ell^{1/3}\}.
	\]
\end{lemma}

\begin{proof}
	We note that, by definition, the event \(E_{\mathrm{st}}(i,\tau)\) is restricted to the super cube \(\prod_{j=1}^d[(i_j-\eta)\ell,(i_j+\eta+1)\ell]\) and time interval \([\tau\beta,(\tau+1)\beta]\).
	We define the following 3 events.
	\begin{itemize}
		\item[\(F_1\):] The initial infected particle \(x_0\) never leaves \(\prod_{j=1}^d[(i_j-\eta+1)\ell,(i_j+\eta-1)\ell]\) during \([\tau\beta,\tau\beta+~T]\).
		\item[\(F_2\):] Let \(C_1\) be the constant from \Cref{lemma:particles2}. During the time interval \([\tau\beta,\tau\beta+T]\) the initial infected particle \(x_0\) collides with at least \(\frac{C_1\lambda_0\ell^{1/3}}{2}\) different particles from \(\Upsilon\) that are in the supercube \(Q^{**}=\prod_{j=1}^d[(i_j-\eta)\ell,(i_j+\eta+1)\ell]\) at time \(\tau\beta+T\).
		\item[\(F_3\):] Out of the \(\frac{C_1\lambda_0\ell^{1/3}}{2}\) or more particles from \(F_2\), at least one of them is in the cube \(\prod_{j=1}^d[(i_j+k_j)\ell,(i_j+k_j+1)\ell]\) at time \((\tau+1)\beta\), for all \(k=(k_1,\dots,k_d)\) for which \(\prod_{j=1}^d[(i_j+k_j)\ell,(i_j+k_j+1)\ell]\subset Q^{**}\).
	\end{itemize}
	By definition of the events, we clearly have that \(\mathbb{P}[E_{\mathrm{st}}(i,\tau)]\geq \mathbb{P}[F_1\cap F_2\cap F_3]\).

Using \Cref{prop:exit} we have
	\begin{equation}\label{for:F1}
		\mathbb{P}[F_1]\geq 1-C_2\exp\{-C_3\ell^2/T\}=1-C_2\exp\{-C_3\ell^{1/3}\}
	\end{equation}
	for some positive constants \(C_2\) and \(C_3\). We observe that \(F_1\) is restricted to the super cube \(\prod_{j=1}^d[(i_j-\eta)\ell,(i_j+\eta+1)\ell]\) and the time interval \([\tau\beta,\tau\beta+T]\).

	For the event \(F_2\), we apply \Cref{lemma:particles2} to get that the intensity of the Poisson point process of particles that are in \(Q^{**}\) at time \(\tau\beta\) and collide with \(x_0\) during \([\tau\beta,\tau\beta+T]\) is at least \(\lambda_0 C_1\ell^{1/3}\) for some positive constant \(C_1\). Since every particle that collides with \(x_0\) enters \(\prod_{j=1}^d[(i_j-\eta+1)\ell,(i_j+\eta)\ell]\) during \([\tau\beta,\tau\beta+T]\), we can use \Cref{prop:exit} to bound the probability that the particle is inside of \(Q^{**}\) at time \(\tau\beta+T\) by
	\[
		1-C_a\exp\left\{-\frac{C_b\ell^2}{T}\right\}=1-C_a\exp\{-C_b\ell^{1/3}\},
	\]
	for some positive constants \(C_a\) and \(C_b\). This term can be made as close to \(1\) as possible by having \(\ell\) sufficiently large. We assume \(\ell\) is large enough so that this term is larger than \(2/3\). This gives that the intensity of the process of particles from \(\Upsilon\) that collided with \(x_0\) during \([\tau\beta,\tau\beta+T]\) and are in \(Q^{**}\) at time \(\tau\beta+T\) is at least 
	\[
		\frac{2\lambda_0 C_1\ell^{1/3}}{3}.
	\]
	Using Chernoff's bound (see \Cref{lem:chernoff}) we have that
	\begin{equation}\label{for:F2}
		\mathbb{P}[F_2]\geq 1-\exp\{-(2/3)^2C_1\lambda_0\ell^{1/3}\}.
	\end{equation}
	Note that, by construction, \(F_2\) is restricted to the super cube \(Q^{**}\) and the time interval \([\tau\beta,\tau\beta+T]\). Furthermore, \(F_2\) is clearly an increasing event. 
	
	We now turn to \(F_3\). Using \Cref{lemma:particles1}, and a uniform bound across the number of cubes inside a super cube, we have that 
	\begin{equation}\label{for:F3}
		\mathbb{P}[F_3]\geq 1-(2\eta+1)^d\exp\{-\frac{C_1\lambda_0\ell^{1/3}}{2}c_p\},
	\end{equation}
	where \(c_p\) is a small but positive constant. Again, the event is restricted to the super cube \(Q^{**}\) and the time interval \([\tau\beta+T,(\tau+1)\beta]\) and is an increasing event. Taking the product of the probability bounds in (\ref{for:F1}), (\ref{for:F2}) and (\ref{for:F3}), we see that the probability that \(E_{\mathrm{st}}(i,\tau)\) holds is at least
	\begin{align*}
		1-	\exp\{-C\lambda_0\ell^{1/3}\}
	\end{align*}
	for some constant \(C\) and all large enough \(\ell\).
\end{proof}

\begin{proof}[Proof of \Cref{thrm:spread}]\label{proof1}
	We start by using \Cref{thrm:exp_tail}. Set \(\eta\in\mathbb{N}\) such that \(\eta\geq d\) and set \(\epsilon=1/2\). Fix the ratio \(\beta/\ell^2\) small enough so that the lower bound for \(w\) is at most \(2\eta+1\), and then set \(w=2\eta+1\).
	Assume \(\ell\) is large enough so that \Cref{prop:event} holds. 

	For each \((i,\tau)\in\mathbb{Z}^{d+1}\), define \(E_{\mathrm{st}}(i,\tau)\) as in \Cref{prop:event}. This event is increasing in the number of particles, is restricted to the super cube \(i\) and time interval \([\tau\beta,(\tau+1)\beta]\), and satisfies
	\[
		\mathbb{P}[E_{\mathrm{st}}(i,\tau)]\geq 1-\exp\{-C\lambda_0\ell^{1/3}\},
	\]
	for some constant \(C\). Hence, letting \(\lambda/2\) stand for the measure \(\frac{\lambda}{2}(x)=\frac{\lambda_0\mu_x}{2}\), we have
	\[
		\log\left(\frac{1}{1-\nu_{E_{\mathrm{st}}}(\tfrac{\lambda}{2},Q_{(2\eta+1)\ell},Q_{(2\eta+1)\ell},\eta\beta)}\right)\geq C\lambda_0\ell^{1/3},
	\]
	which increases with \(\ell\), as does the term \(\epsilon^2\lambda_0\ell^d\) in the condition of \Cref{thrm:exp_tail}.
	Thus, setting \(\ell\) large enough, we apply \Cref{thrm:exp_tail} which gives the existence of a two-sided Lipschitz surface \(F\), on which the event \(E_{\mathrm{st}}(i,\tau)\) holds. We also get that the surface is almost surely finite and that it surrounds the origin.
	
	We now proceed to argue that the existence of the surface \(F\) implies that the infection spreads with positive speed.
	Since the two-sided Lipschitz surface \(F\) is finite and surrounds the origin, we have that in almost surely finite time, an infected particle started from the origin will enter some cube \(\prod_{j=1}^d[i_j\ell,(i_j+1)\ell]\) for which \((i,\tau)\) is in \(F\). We call this the \emph{central cube} of \((i,\tau)\). Once that holds, the starting assumption of \(E_{\mathrm{st}}(i,\tau)\) from \Cref{prop:event} is satisfied for the super cell \((i,\tau)\), and the event \(E_{\mathrm{st}}(i,\tau)\) holds. By the definition of \(E_{\mathrm{st}}(i,\tau)\) this means that the initial infected particle for the super cell \((i,\tau)\) infects a large number of other particles, which spread the infection to the central cube of \((i',\tau+1)\) for all \(i'\in\mathbb{Z}^d\) such that \(\|i'-i\|_{\infty}\leq \eta\).
	
	Let \((b,h)\) be the base-height index of the cell \((i,\tau)\in F\). Recall that \(h\) is one of the spatial dimensions. We will also select one of the \(d-1\) spatial dimensions from \(b\) and denote it \(b_1\). Let \(b'\in\mathbb{Z}^d\) be obtained from \(b\) by increasing the time dimension from \(\tau\) to \(\tau+1\), and by increasing the chosen spatial dimension from \(b_1\) to \(b_1+1\). Since \(\|b-b'\|_1=2\), we can choose \(h'\in\mathbb{Z}\) such that \((b',h')\in F\) and \(|h-h'|\leq 2\), where the latter holds by the Lipschitz property of \(F\). Therefore, there must exists \(i'\in\mathbb{Z}^d\) such that \((i',\tau+1)\) is the space-time super cell corresponding to \((b',h')\) and \(\|i-i'\|_{\infty}\leq 1\). Hence, at time \((\tau+1)\beta\), there is an infected particle in the central cube of the super cell \(i'\).
	
	We can then recursively repeat this procedure for the super cell \((i',\tau+1)\), since \(E_{\mathrm{st}}(i',\tau+1)\) holds. Repeating this process we obtain that the infection spreads by a distance of at least \(\ell\) in time \(\beta\) in the chosen spatial direction. Consequently
	\[
		\liminf_{t\rightarrow\infty}\frac{\|I_t\|_1}{t}>0\quad \textrm{almost surely}.
	\]
\end{proof}

In order to prove \Cref{thrm:spread_recov}, we can follow the same steps as in the proof of \Cref{thrm:spread} with the additional consideration that we have to ensure that the relevant infected particles do not recover too quickly. For that, we will require that all the particles involved do not recover for at least \(\beta\).

\begin{proof}[Proof \Cref{thrm:spread_recov}]
	Recall the definition of \(\Upsilon\) and \(\rho\) from \Cref{lemma:particles2} and of \(E_{\mathrm{st}}(i,\tau)\) from \Cref{prop:event}. Let \(E_{\mathrm{st}}'(i,\tau)\) be the event that \(E_{\mathrm{st}}(i,\tau)\) holds, and that the particles in \(\Upsilon\) and the initial infected particle whose path is \(\rho\) do not recover during \([\tau\beta,(\tau+1)\beta]\). Since each such particle does not recover during \([\tau\beta,(\tau+1)\beta]\) with probability \(\exp\{-\gamma\beta\}\), for \Cref{lemma:particles2} we consider that for each \(x\in Q^{*}\setminus\rho(\tau\beta)\) the number of particles at \(x\) at time \(\tau\beta\) that do not recover during \([\tau\beta, (\tau+1)\beta]\) is a Poisson random variable of intensity \(\frac{\lambda_0}{2}\mu_x\exp\{-\lambda\beta\}\). Thus,
	 once \(\eta\), \(\beta\) and \(\ell\) are fixed, setting \(\gamma\) small enough gives that \(E'_{\mathrm{st}}(i,\tau)\) holds with probability at least 
	\[
		1-(1-\exp\{-\gamma\beta\})-\exp\{-C\lambda_0\exp\{-\gamma\beta\}\ell^{1/3}\}
	\]
	for some positive constant \(C\), where the term inside the parenthesis accounts of the probability that the initial infected particles recovers during \([\tau\beta,(\tau+1)\beta]\).
	We now follow the same steps as in the proof of \Cref{thrm:spread} to get that the two-sided Lipschitz surface \(F\) on which the increasing event \(E'_{\mathrm{st}}(i,\tau)\) holds exists, is finite and surrounds the origin almost surely. This gives that an initially infected particle that is at the origin at time \(0\) has a strictly positive probability of surviving long enough to enter a cell of the two-sided Lipschitz surface. Once on the surface, the infection survives indefinitely by the definition of \(E'_{\mathrm{st}}(i,\tau)\). Hence
	\[
	\mathbb{P}\left[\|I_t\|_1\geq c_1t\textrm{ for all }t\geq0\right]\geq c_2.
	\]
\end{proof}

\appendix
\section{Appendix: Standard large deviation results}
\begin{lemma}[Chernoff bound for Poisson]\label{lem:chernoff}
Let \(P\) be a Poisson random variable with mean \(\lambda\). Then, for any \(0<\epsilon<1\),
\[
	\mathbb{P}[P<(1-\epsilon)\lambda] < \exp\{-\lambda\epsilon^2/2\}
\]
and
\[
	\mathbb{P}[P > (1 + \epsilon)\lambda] < \exp\{-\lambda\epsilon^2/4\}.
\]
\end{lemma}

\bibliographystyle{plain}\addcontentsline{toc}{section}{References}
\bibliography{library}

\begin{thebibliography}{10}

\bibitem{Andres2016}
Sebastian Andres, Jean~Dominique Deuschel, and Martin Slowik.
\newblock {Harnack inequalities on weighted graphs and some applications to the
  random conductance model}.
\newblock {\em Probability Theory and Related Fields}, 164(3-4):931--977, 2016.

\bibitem{Barlow2004}
Martin~T. Barlow.
\newblock {Random walks on supercritical percolation clusters}.
\newblock {\em Annals of Probability}, 32(4):3024--3084, 2004.

\bibitem{Hambly2009}
Martin~T. Barlow and Ben~M. Hambly.
\newblock {Parabolic Harnack inequality and local limit theorem for percolation
  clusters}.
\newblock {\em Electronic Journal of Probability}, 14:1--26, 2009.

\bibitem{Berger2008}
Noam Berger, Marek Biskup, Christopher~E. Hoffman, and Gady Kozma.
\newblock {Anomalous heat-kernel decay for random walk among bounded random
  conductances}.
\newblock {\em Annales de l'institut Henri Poincare (B) Probability and
  Statistics}, 44(2):374--392, 2008.

\bibitem{Delmotte1999}
Thierry Delmotte.
\newblock {Parabolic Harnack inequality and estimates of Markov chains on
  graphs}.
\newblock {\em Revista Matematica Iberoamericana}, 15(1):181--232, 1999.

\bibitem{Gracar2016a}
Peter Gracar and Alexandre Stauffer.
\newblock {Multi-scale Lipschitz percolation of increasing events for Poisson
  random walks}.
\newblock {\em arXiv}, 1702.08748, 2017.

\bibitem{Hilario2014}
Marcelo Hil{\'{a}}rio, Frank den Hollander, Vladas Sidoravicius, Renato {Soares
  dos Santos}, and Augusto Teixeira.
\newblock {Random walk on random walks}.
\newblock {\em Electronic Journal of Probability}, 20, 2015.

\bibitem{Kesten2005}
Harry Kesten and Vladas Sidoravicius.
\newblock {The spread of a rumor or infection in a moving population}.
\newblock {\em Annals of Probability}, 33(6):2402--2462, 2005.

\bibitem{Kesten2008}
Harry Kesten and Vladas Sidoravicius.
\newblock {A shape theorem for the spread of an infection}.
\newblock {\em Annals of Mathematics}, 167(3):701--766, 2008.

\bibitem{Peres2012}
Yuval Peres, Alistair Sinclair, Perla Sousi, and Alexandre Stauffer.
\newblock {Mobile geometric graphs: Detection, coverage and percolation}.
\newblock {\em Probability Theory and Related Fields}, 156(1-2):273--305, 2013.

\bibitem{Popov2012}
Serguei Popov and Augusto Teixeira.
\newblock {Soft local times and decoupling of random interlacements}.
\newblock {\em Journal of the European Mathematical Society},
  17(10):2545--2593, 2015.

\bibitem{Sinclair2010}
Alistair Sinclair and Alexandre Stauffer.
\newblock {Mobile geometric graphs, and detection and communication problems in
  mobile wireless networks}.
\newblock {\em arXiv}, 1005.1117, 2010.

\bibitem{Stauffer2014}
Alexandre Stauffer.
\newblock {Space-time percolation and detection by mobile nodes}.
\newblock {\em Annals of Applied Probability}, 25(5):2416--2461, 2015.

\end{thebibliography}
\end{document}